\documentclass[11pt, a4paper]{article}
\usepackage{mathrsfs, mathtools, amsthm, amssymb, thm-restate}  % Packages for math; put "thm-restate" ahead of "hyperref" so that it's unnecessary to define "\lemmaautorefname"
\allowdisplaybreaks[4]
\usepackage{paralist, tablists}  % Inline and compact lists (place before enumitem)
\usepackage{enumerate}
\usepackage{booktabs}
\usepackage{autobreak}

\usepackage{graphicx, xcolor, tikz}  % Packages for figures
\usetikzlibrary{calc, shapes, decorations.pathreplacing, decorations.markings}  % TikZ libraries for path decorations
\usepackage{caption}  % Custom captions
\usepackage[labelformat=simple]{subcaption}  % Subfigure captions

\usepackage[left=25mm, top=25mm, bottom=25mm, right=25mm]{geometry}  % Set page margins
\usepackage[T1]{fontenc} % https://tex.stackexchange.com/questions/664/why-should-i-use-usepackaget1fontenc

\usepackage[inline]{enumitem}  % Inline lists
\usepackage[square, numbers, sort&compress]{natbib}  % Citation style (square brackets, numbers, and sorting)
\usepackage[
  bookmarks=true,                % Enable bookmarks
  bookmarksnumbered=true,        % Include section numbers in bookmarks
  bookmarksopen=true,            % Expand the bookmarks tree by default
  plainpages=false,              % Don't use plain page numbers in the bookmarks
  colorlinks=true,               % Enable colored text links (instead of boxes)
  linkcolor=blue,                % Link color for internal references (sections, table of contents)
  citecolor=black,               % Citation link color
  anchorcolor=green,             % Color of anchors (used for references)
  urlcolor=blue,                 % Link color for URLs
  hyperindex=true                % Enable clickable index links
]{hyperref} % Always load hyperref last, unless other packages require exceptions，不能和cite宏包同时使用

\newtheorem{thm}{Theorem}[section] % \newtheorem一定要放在hyperref后边，否则可能造成\autoref及\cref的指向位置不正确
\newtheorem{lem}[thm]{Lemma}

\newtheorem{cor}[thm]{Corollary}
\newtheorem{conj}{Conjecture}
\newtheorem{case}{Case}
\AtEndEnvironment{proof}{\setcounter{case}{0}}

\newtheorem{claim}{Claim}
\theoremstyle{definition}

\makeatletter
\def\th@plain{%
  \upshape %\itshape % Use upright font in the body of the theorem
}
\makeatother

\usepackage{array}
\usepackage{bm}  % the package "bm" and "newtxmath" are conflict in MacTex 2020, so we must use "bm" later

\makeatletter
\renewenvironment{proof}[1][\proofname]{\par
  \pushQED{\qed}%
  \normalfont \topsep6\p@\@plus6\p@\relax
  \trivlist
  \item[\hskip\labelsep
        \bfseries
    #1\@addpunct{.}]\ignorespaces
}{%
  \popQED\endtrivlist\@endpefalse
}
\makeatother

\usepackage[capitalise]{cleveref}  % Automatic references for theorem-like environments
\crefname{claim}{Claim}{Claims}
\crefname{problem}{Problem}{Problems}
\crefname{conjecture}{Conjecture}{Conjectures}
\usepackage{algorithm}
\usepackage[noend]{algpseudocode}  % noend表示算法不显示 EndIf或者EndFor这些，可以用来节省空间
\usepackage{algorithmicx}

\providecommand{\sep}{, }   % 定义关键词分隔符
\newenvironment{keyword}
  {\par\noindent\textbf{Keywords:}\enspace}
  {\par\vspace{\baselineskip}}
\begin{document}

\title{An Improved Interpolation Theorem and Disproofs of Two Conjectures on 2-Connected Subgraphs}

\author{
    Haiyang Liu\thanks{College of Cryptology and Cyber Science, Nankai University, Tianjin 300350, P.R. China.}
    \and 
    Bo Ning\thanks{Corresponding author. College of Computer Science, Nankai University, Tianjin 300350, P.R. China. E-mail: \texttt{bo.ning@nankai.edu.cn} (B. Ning). Partially supported by the National Nature Science Foundation of China (No. 12371350). \href{https://orcid.org/0000-0002-9622-5567}{ORCID: 0000-0002-9622-5567}.}}
\maketitle

\begin{abstract}
We prove that any \(2\)-connected graph \(G\) on \(n\) vertices with minimum degree \(\delta(G) \ge \frac{n}{4}+2\) contains a \(2\)-connected subgraph of order \(k\) for every integer \(k\) with \(4 \le k \le n\). This improves a previous result of Yin and Wu. In \cite{YinWu-DAM-2026}, Yin and Wu proposed two conjectures. The first states that for any \(2\)-connected graph \(G\) of order \(n\) and size \(m\), there exists a \(2\)-connected subgraph of order \(k\) for each \(k \in \{4, \dots, n\}\) whenever \(m \ge \frac{1}{2} n^{3/2}\). The second conjecture asserts that the same conclusion holds under the alternative condition \(\delta(G) \ge \sqrt{n}\). In this paper, we construct counterexamples that completely disprove the first conjecture. Furthermore, using the existence of \((v, k, 2)\)-Symmetric Balanced Incomplete Block designs (i.e., SBIBDs), we disprove the second conjecture for all \(n \in \{8, 14, 22, 32, 74, 112, 158\}\). Finally, we propose a conjecture of our own: for any \(2\)-connected graph \(G\) on \(n\) vertices with \(\delta(G) \ge \frac{n}{k}\), where \(k \ge 3\) and \(n\) is sufficiently large, \(G\) contains a \(2\)-connected subgraph of every order from \(4\) to \(n\).
\end{abstract}
%% Keywords


\section{Introduction}
\label{ch:1}
In this article, all graphs under consideration are simple, undirected, finite and refer to Bondy and Murty~\cite{MurtyBondy-635} for the undefined notation and concepts.

A graph invariant is said to possess the \emph{interpolation property}over a family of subgraphs if, whenever two members of the family take values 
$m$ and $n$
for that invariant (with $m<n$), then for every integer 
$k$ between $m$
and $n$, there exists a member of the family for which the invariant equals exactly $k$.
This property reveals a certain ``continuity" or ``completeness" in the range of attainable values, and has significant implications for understanding the structural diversity within families of subgraphs.

% Let $\mathcal{H}$ be a class of graphs, and let $G$ be a graph. If there exist two integers $r_1$ and $r_2$ with $r_1
% \leqslant r_2$ such that $G$ contains all subgraphs $H\in\mathcal{H}$ of order $\ell$ for any integer $\ell$ satisfying $r_1\leq \ell\leq r_2$, then we say $G$ has the interpolation property for $\mathcal{H}$. 

 The concept of interpolation theorems in graph theory has deep roots in the study of spanning trees and their invariants. 
 Among the various invariants studied for spanning trees, the number of end-vertices (usually referred to as leaves) occupies a central position. A leaf in a tree is a vertex of degree one, and its count is one of the most fundamental parameters describing the shape and complexity of a tree. The question of whether the leaf counts of spanning trees of a given graph can realize all intermediate values was formally posed by Chartrand at the Fourth International Conference on Graph Theory and Applications held in Kalamazoo in 1980.
This problem was resolved by Seymour Schuster \cite{schuster1983interpolation} in 1983, who proved the following landmark result.

\begin{thm}[Schuster \cite{schuster1983interpolation}]
Let \(G\) be a connected graph. If \(G\) contains spanning trees having exactly \(m\) and \(n\) end-vertices, with \(m < n\), then for every integer \(k\) satisfying \(m < k < n\), \(G\) also contains a spanning tree having exactly \(k\) end-vertices .   
\end{thm}

The interpolation property for cycles has been a subject of investigation. A graph \( G \) of order \( n \) is said to be pancyclic if it contains cycles of every length \( k \) from \( 3 \) to \( n \). The first degree-based condition for pancyclicity was established by Bondy \cite{BONDY197180}, which we restate here in terms of the interpolation property of cycles.

\begin{thm}[Bondy~\cite{BONDY197180}]
Let $G$ be a graph of order $n$ and size $m$. Suppose that $m \geqslant n^2/4$. If $G$ contains a $C_n$ then $G$ contains a cycle $C_k$ for any integer $3\leq k\leq n$, unless $n$ is even and $G\cong K_{\frac{n}{2},\frac{n}{2}}$.
\end{thm}

 \begin{cor}
 For a graph G of order $n$, if $\delta(G) > \left\lceil\frac{ n }{2}\right\rceil$, then $G$ contains all
possible order of cycles.
 \end{cor}

 Interpolation properties for some classes of graphs have been extensively studied in the literature \cite{harary1967interpolation,schuster1983interpolation,fisher1997number,harary1989classification,lih2006interpolation,topp1998interpolation}.

Observe that a cycle $C_k$ itself is a $2$-connected graph of order $k$. By replacing the requirement of a cycle of length $k$ with that of a $2$-connected subgraph of order $k$, one may expect that the degree condition $\delta(G)\ge n/2$ in Bondy's classic theorem can be relaxed. Motivated by this idea, Yin and Wu \cite{YinWu-DAM-2026} recently investigated the interpolation property for $2$-connected subgraphs. They proved that every $2$-connected graph $G$ on $n$ vertices with minimum degree $\delta(G)\ge n/3+1$ contains a $2$-connected subgraph of order $k$ for any integer $k$ with $4\le k\le n$.

\begin{thm}(\rm {Yin and Wu~\cite[Theorem~3.6]{YinWu-DAM-2026}})\label{Thm:YinWu}
  Let $G$ be a $2$-connected graph of order $n \geqslant 4$. If $\delta (G)
  \geqslant \left\lceil \frac{n}{3} \right\rceil + 1$, then $G$ contains a $2$-connected
  subgraph of order $k$ for any $k \in \{ 4, \ldots, n \} .$\label{MinDegn3+1}
\end{thm}

They also posed the following conjectures.

\begin{conj}(\rm {Yin and Wu \cite[Conjecture~4.1]{YinWu-DAM-2026}})\label{Conj:1}
  Let $G$ be a $2$-connected graph of order $n$ and size $m$. If $m \geqslant
  \frac{1}{2} n^{\frac{3}{2}}$ then $G$ has $2$-connected subgraph of order
  $k$ for each $k \in \{ 4, \ldots, n \}$.\label{MinSize3/2}
\end{conj}

The next conjecture appears in both the abstract and the end of the introduction of \cite{YinWu-DAM-2026}.
\begin{conj}[Yin and Wu \cite{YinWu-DAM-2026}]\label{Conj:2}
  Let $G$ be a $2$-connected graph of order $n$. If $\delta (G) \geqslant
  \sqrt{n}$, then $G$ has a $2$-connected subgraph of order $k$ for each $k \in
  \{ 4, \ldots, n \}$.\label{MinDegSqrt}
\end{conj}

In this article, we disprove Conjecture \ref{Conj:1} and disprove Conjecture \ref{Conj:2}   for some small orders using the existence of $(v,k,2)-$SBISDs (see Section \ref{Couters}). Moreover, we improve the minimum degree condition of Theorem \ref{Thm:YinWu} as follows.

\begin{thm}
  Let $G$ be a $2$-connected graph of order $n \geqslant 4$. If $\delta (G)
  \geqslant \left\lceil \frac{n}{4} \right\rceil + 2$, then $G$ contains a $2$-connected
  subgraph of order $k$ for any $k \in \{ 4, \ldots, n \}$.\label{MyMinDeg}
\end{thm}

We introduce some necessary notation. Let $G$ be a graph. If $X$ is a vertex subset of graph $G$, then the subgraph
of $G$ induced by $X$ is denoted by $G [X]$. We denote by $G-S$ the induced subgraph $G[V(G)\backslash S]$ of $G$. Usually, for a vertex $v\in V(G)$, we use $G-v$
instead of $G-\{v\}$. For a subgraph $H$ of $G$, we use $G-H$ instead of $G-V(H)$ for sometime. We denote by $\delta(G)$ the minimum degree of $G$.
The cartesian product $G\Box H$
of two graphs $G$ and $H$ is the graph with $V (G\Box H) = V (G)\times V (H)$, in
which two vertices $(g_1, h_1)$ and $(g_2, h_2)$ are adjacent if and only if
either $u_1 u_2 \in E (G)$ and $v_1 = v_2$, or $v_1 v_2 \in E (H)$ and $u_1 =
u_2$. $C_k$ and $K_{a,b}$ denote a cycle with $k$ vertices and a complete bipartite with two parts of $a$ and $b$ vertices, respectively.

In Section \ref{Couters}, we construct some counterexamples to
Conjecture \ref{MinSize3/2} and 
Conjecture \ref{MinDegSqrt}. In Section \ref{Refinement}, we prove Theorem \ref{MyMinDeg}. In Section \ref{ExampleForTheMin}, we construct
some examples to show that the condition of Theorem \ref{MinDegn3+1} is sharp
when $n$ is small and conjecture that the condition of Theorem \ref{MyMinDeg} can be replaced with $\delta(G)\geq \frac{n}{k}$ for $n$ sufficiently large and $k\geq 3$.

\section{Counterexamples to Conjectures
\ref{MinSize3/2} and \ref{MinDegSqrt}}\label{Couters}

In this section, we first construct a counterexample to disprove Conjecture
\ref{MinSize3/2}. Let $\epsilon < \frac{1}{4}$ be a positive real number, and $n$
be an integer. We form a graph $G_{\epsilon}$ of order $n$ from a complete
graph $K_{_{\left\lfloor n^{1 - \epsilon} \right\rfloor}}$ of order $\left\lfloor n^{1 -
\epsilon} \right\rfloor$ and a path $P_{n - \left\lfloor n^{1 - \epsilon} \right\rfloor}$ of
order $n - \left\lfloor n^{1 - \epsilon} \right\rfloor$ by adding two vertex-disjoint
edges $e_1, e_2$ as follows. Let $x', y'$ be the endpoints of $P_{n - \left\lfloor
n^{1 - \epsilon} \right\rfloor}$, and choose two distinct vertices $x, y \in V
\left( K_{_{\left\lfloor n^{1 - \epsilon} \right\rfloor}} \right)$. We let $e_1$ join $x$
to $x'$, and $e_2$ join $y$ to $y'$ (see Fig. \ref{pasting}).

\begin{figure}[ht]
  \centering
  \raisebox{-0.4747405480575104\height}{\includegraphics[scale=0.7]{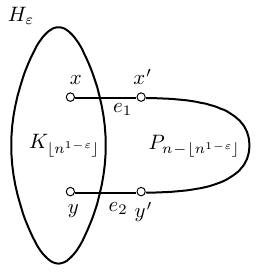}}
  \caption{A counterexample to Conjecture
\ref{MinSize3/2}}
  \label{pasting}
\end{figure}

Firstly, we count the size of $G_{\epsilon}$ as follows:
\begin{align*}
  e (G_{\epsilon})&= e( K_{_{\left\lfloor n^{1 - \epsilon} \right\rfloor}}) + e (P_{n - \left\lfloor n^{1 - \epsilon} \right\rfloor})\\
  &=\frac{1}{2} \left\lfloor n^{1 - \epsilon} \right\rfloor (\left\lfloor n^{1 -
  \epsilon} \right\rfloor - 1) + (n - \left\lfloor n^{1 - \epsilon} \right\rfloor + 1)\\
  & \geqslant  \frac{1}{2} (n^{1 - \epsilon} - 1) (n^{1 - \varepsilon} - 2)
  + (n - n^{1 - \varepsilon} + 1)\\
  & =  \frac{1}{2} n^{2 - 2 \varepsilon} + n - \frac{5}{2} n^{1 -
  \varepsilon} + 2\\
  & =  \frac{1}{2} n^{2 - 2 \varepsilon} + O (n).
\end{align*}
Since $\varepsilon < \frac{1}{4}$, $e (G_{\varepsilon}) > \frac{1}{2}
n^{\frac{3}{2}}$ for $n$ is sufficiently large.

Next, we count all the orders of $2$-connected subgraphs of $G_{\varepsilon}$.

\begin{thm}
$G_{\varepsilon}$ contains a $2$-connected subgraph of order $k$ for each $k
  \in \{ 3, 4, \ldots, \left\lfloor n^{1 - \epsilon} \right\rfloor \} \cup \{ n - \left\lfloor
  n^{1 - \epsilon} \right\rfloor + 2, n - \left\lfloor n^{1 - \epsilon} \right\rfloor + 3,
  \ldots, n \}$.
\label{SubgraphOfHe}
\end{thm}

\begin{proof}
Let $G'$ be a $2$-connected subgraph of $G_{\varepsilon}$. Suppose that $G'$ is a subgraph of $K_{_{\left\lfloor n^{1 - \epsilon} \right\rfloor}}$. Since $K_{_{\left\lfloor
  n^{1 - \epsilon} \right\rfloor}}$ is a complete subgraph of $G_{\varepsilon}$, it
  is easy to check that the order of $G'$ can be taken in $\{ 3, \ldots,
  \left\lfloor n^{1 - \epsilon} \right\rfloor \}$.
  
      If $G'$ contains at least one vertex in $P_{n - \left\lfloor n^{1 - \epsilon}
  \right\rfloor}$, then $G'$ must contain $x, y$ and all vertices of $P_{n -
  \left\lfloor n^{1 - \epsilon} \right\rfloor}$. Thus, the order of $G'$ is at least $n -
  \left\lfloor n^{1 - \epsilon} \right\rfloor + 2$. If the order of $G'$ is $n - \left\lfloor
  n^{1 - \epsilon} \right\rfloor + 2$, then $G' \cong G_{\varepsilon} [V (P_{n -
  \left\lfloor n^{1 - \epsilon} \right\rfloor}) \cup \{ x, y \}]$. Let $\ell \in [n -
  \left\lfloor n^{1 - \epsilon} \right\rfloor + 3, n]$ be an integer. \ Since
  $K_{_{\left\lfloor n^{1 - \epsilon} \right\rfloor}}$ is a complete subgraph of
  $G_{\varepsilon}$, we can form a $2$-connected subgraph $G_{\ell} \cong
  G_{\varepsilon} [V (G') \cup W]$ of order $\ell$ in $G_{\varepsilon}$ from
  $G'$ and a vertex subset $W \subseteq V \left( K_{_{\left\lfloor n^{1 - \epsilon}
  \right\rfloor}} \right) \backslash \{ x, y \}$ of order $\ell - n + \left\lfloor n^{1 -
  \epsilon} \right\rfloor - 2$. Thus, the order of $G'$ can be taken in $\{ n -
  \left\lfloor n^{1 - \epsilon} \right\rfloor + 2, n - \left\lfloor n^{1 - \epsilon} \right\rfloor +
  3, \ldots, n \}$.
  
Therefore, $G_{\varepsilon}$ contains a $2$-connected subgraph of order $k$ for each $k \in \{ 3, 4, \ldots, \left\lfloor n^{1 - \epsilon} \right\rfloor \} \cup \{
  n - \left\lfloor n^{1 - \epsilon} \right\rfloor + 2, n - \left\lfloor n^{1 - \epsilon}
  \right\rfloor + 3, \ldots, n \}$. The proof is complete.
\end{proof}

Using Theorem \ref{SubgraphOfHe}, it is easy to check that 
$G_{\varepsilon}$ does not satisfy the conclusion of Conjecture \ref{MinSize3/2}. Since
$\varepsilon < \frac{1}{4}$, $n - \left\lfloor n^{1 - \epsilon} \right\rfloor + 2 >
\left\lfloor n^{1 - \epsilon} \right\rfloor + 1$ if $n \geqslant 16$. Therefore, the
graph $G_{\varepsilon}$ for $n$ sufficiently large is a counterexample to
Conjecture \ref{MinSize3/2}. Moreover, for any $\varepsilon' < 2$, if we
modify the condition of Conjecture \ref{MinSize3/2} to the condition ``$m
\geqslant \frac{1}{2} n^{2 - \varepsilon'}$'', we can also construct a
counterexample by setting $\varepsilon < \frac{1}{2} \varepsilon'$ and
applying the same method described in this section.

A \textbf{design} is a pair $(X,\mathcal{A})$ such that the following properties are satisfied: (i) $X$ is a set of elements called points. (ii) $\mathcal{A}$ is a collection of nonempty subsets of X called blocks. 
A design is called a \textbf{Balanced Incomplete Block design} (BIBD) 
if $v$ points are arranged in $b$ blocks, each block containing $k (<v)$ points, every point occurs at the most once in a block and in exactly $r$ blocks ensuring that every pair of symbols occurs together in $\lambda$ blocks.
A BIBD with $b = v$ (or equivalently, $r = k$) is a  $(v,k,\lambda)$-\textbf{Symmetric Balanced Incomplete Block design} (i.e., $(v,k,\lambda)$-SBIBD). 

The following fact comes from Remark 6.14 in \cite[pp.~111]{CD06}.

\begin{lem}[Colbourn and Dinitz~\cite{CD06}]
\label{SBIBD_exist}
If $\lambda=2$, the only values of $k$ for which $(v,k,2)$-
SBIBD are known are $k = 4, 5, 6, 9, 11, 13$. 
\end{lem}

\begin{thm}\label{Thm:2.2}
The statement ``Let $G$ be an $n$-vertex graph. If $\delta(G)\geqslant \sqrt{n}$ then $G$ contains a $C_5$ or a $K_{2,3}$"  
is false for each $n\in \{8,14,22,32,74,112,158\}$.
\end{thm}
\begin{proof}
Consider the incident graph $G_n$ of $(v,k,2)-$SBISD. Then $G_n$ is a bipartite graph and hence contains no $C_5$. For any two vertices in $G_n$, they have either 2 or 0 common neighbors,
and hence contain no $K_{2,3}$. By Lemma \ref{SBIBD_exist}, $(v,k,2)-$SBISD  exists
for each
$k\in \{4,5,6,9, 11,13\}$.

It is known that $v=\frac{k(k-1)}{2}+1$ and $v(G_n)=2v=k^2-k+2\in \{14,22,32,74,112,158\}$.
Thus, $\delta(G)=k>\sqrt{v(G_n)}$. For $n=8$, we consider a hypercube graph $Q_3 = C_4 \Box K_2$ (see
Fig. \ref{HypergraphQ3}). Then $\delta (Q_3) = 3 > \sqrt{8}$.
\end{proof}

\begin{figure}[ht]
  \centering
  \raisebox{-0.4747405480575104\height}{\includegraphics[scale=0.7]{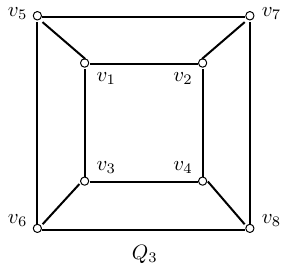}}
  \caption{A hypercube $Q_3=C_4 \Box K_2$}
  \label{HypergraphQ3}
\end{figure}

One can see that if Conjecture \ref{Conj:2} is true, the graph $G$ surely contains a 2-connected subgraph of order 5, and hence contains a $C_5$ or a $K_{2,3}$. So Theorem \ref{Thm:2.2} immediately has the following,
which shows that Conjecture
\ref{MinDegSqrt} is false for a few of small values of $n$. 

\begin{cor}\label{Q3Subgraph}
Conjecture \ref{Conj:2} is false for each $n\in \{8,14,22,32,74,112,158\}$. 
\end{cor}

\section{Proof of Theorem \ref{MyMinDeg}}\label{Refinement}

We will use the following  results.

\begin{thm}[Reiman~\cite{Reiman-702}]
  Any simple $n$-vertex graph $G$ with $\sum_{v \in V} 
  \binom{d (v)}{2}> \binom{n}{2}$ contains a cycle of length $4$.\label{reiman}
\end{thm}

A graph $G$ is critically $h$-connected if the connectivity $\kappa(G)=h$ 
and $\kappa(G-v)<h$ for any vertex $v\in V(G)$.

\begin{thm}(\rm {Hamidoune~\cite[Theorem 3.5]{HAMIDOUNE1980257}})
Every critically $h$-connected graph $G$ contains
two vertices of degree not exceeding $\frac{3h}{2}-1$.
\end{thm}

\begin{cor}\label{CritMin}
In any critically $2$-connected graph, there are at least two vertices of
degree $2$.
\end{cor}

\begin{lem}(\rm {Yin and Wu \cite[Lemma~3.5]{YinWu-DAM-2026}})
  Let $G$ be a $2$-connected graph of order $n \geqslant 5$ and $V' \subseteq
  V (G)$ with $2 \leqslant | V' | \leqslant 3$. If $d (v) \geqslant \frac{n +
  3}{2}$ for any vertex $v \in V (G) \backslash V'$, then there exist two
  vertices $v_1$ and $v_2$ such that $| V' \backslash \{ v_1, v_2 \} |
  \geqslant 2$ and $G -\{v_1,v_2\}$ satisfies one of the following additional
properties:\label{Removable}
  \begin{enumerate}
    \item[(i)] $G -\{v_1,v_2\}$ is still $2$-connected;\label{keep}
    
    \item[(ii)] $G -\{v_1,v_2\}$ has exactly two blocks, one of which is isomorphic
    to $K_2$ whose the pendent vertex belongs to $V'$ and the cut vertex do
    not belong to $V'$; and \label{TwoBlock}
    
    \item[(iii)] $G -\{v_1,v_2\}$ consists of exactly three blocks with two cut
    vertices, in which the two end blocks are isomorphic to $K_2$ each and the
    two pendent vertices belongs to $V'$, but neither of the two cut vertices
    belong to $V'$.\label{ThreeBlock}
  \end{enumerate}
\end{lem}

Now we begin to prove Theorem \ref{MyMinDeg}.

\begin{proof}[Proof of Theorem \ref{MyMinDeg}]
  For $4\leqslant n \leqslant 15$, $\delta (G) \geqslant \left\lceil \frac{n}{4} \right\rceil + 2$
  implies that $\delta (G) \geqslant \left\lceil \frac{n}{3} \right\rceil + 1$. By
  Theorem \ref{MinDegn3+1}, the result holds. So, we assume that $n \geqslant
  16$ in what follows.
  
Since
\begin{align*}
    \sum_{v \in V}
    \binom{d (v)}{2} \geqslant & n \binom{\frac{n}{4} + 2}{2} \\
    = & \frac{1}{2} n \left( \frac{n}{4} + 2 \right) \left( \frac{n}{4} + 1
    \right) - \frac{1}{2} n (n - 1) + \binom{n}{2} \\
    >& \frac{1}{32} n (n - 2)^2 + \binom{n}{2} \\
    > & \binom{n}{2},
  \end{align*}
by Theorem \ref{reiman}, $G$ contains a $C_4$, i.e., a $2$-connected subgraph
  of order $4$.
  
  Next, let $k \geqslant 5$, and we shall show that $G$ contains a $2$-connected
  graph of order $k$ by induction on $k$. Let $H_{k - 1}$ be a $2$-connected
  subgraph of $G$ with order $k - 1$. If there is a vertex $x \in V (G)
  \backslash V (H_{k - 1})$ such that $| N_G (x) \cap V (H_{k - 1}) |
  \geqslant 2$, then $G [V (H_{k - 1}) \cup \{ x \}]$ is a $2$-connected
  subgraph, say $H_k$, of order $k$ in $G$. Therefore, we always assume that there
  does not exist such a vertex in $V (G) \backslash V (H_{k - 1})$, i.e.
  \begin{equation}
    | N_G (x) \cap V (H_{k - 1}) | \leqslant 1 \text{ for any } x \in V (G)
    \backslash V (H_{k - 1}) . \label{NeighLess1}
  \end{equation}
  It implies that
  \begin{equation}
    e (V (H_{k - 1}), V (G) \backslash V (H_{k - 1})) \leqslant n - k + 1.
    \label{EdgeUpper}
  \end{equation}
  Since $\delta (G) \geqslant \left\lceil \frac{n}{4} \right\rceil + 2$ and by the assumption
  \eqref{NeighLess1}, we have
  \begin{equation}
    \delta (G - V (H_{k - 1})) \geqslant \left\lceil \frac{n}{4} \right\rceil + 1.
    \label{MinDegHk-1}
  \end{equation}
  \begin{case}\label{Case:1}
    $5 \leqslant k \leqslant \left\lceil \frac{n}{4} \right\rceil$.
  \end{case}
  Since $\delta (G) \geqslant \left\lceil \frac{n}{4} \right\rceil + 2$, for each vertex
  $v$ of $H_{k - 1}$,
  \[ e (v, V (G) \backslash V (H_{k - 1})) = | N_G (v) \backslash V (H_{k -
     1}) | \geqslant \left\lceil \frac{n}{4} \right\rceil + 2 - (k - 2) \geqslant
     \frac{n}{4} - k + 4. \]
  Since $5 \leqslant k \leqslant \left\lceil \frac{n}{4} \right\rceil$, $(k - 5) \left( k - \frac{n}{4}
  - 1 \right) \leqslant 0$, and thus
  \begin{align*}
    e (V (H_{k - 1}), V (G) \backslash V (H_{k - 1})) &=  \sum_{v \in V
    (H_{k - 1})} | N_G (v) \backslash N_{H_{k - 1}} (v) |\\
    & \geqslant  (k - 1) \left( \frac{n}{4} - k + 4 \right)\\
    & =  n - k + 1 - (k - 5) \left( k - \frac{n}{4} - 1 \right)\\
    & \geqslant  n - k + 1,
  \end{align*}
  where all equalities hold if and only if one of the following holds: (a) $k = 5$, $n \bmod 4 \equiv 0$, and $e (v, V (G)
  \backslash V (H_4)) = \frac{n}{4} - 1$ holds for each $v \in V (H_4)$; (b) $k=\frac{n}{4}+1$, $n \bmod 4 \equiv 0$, and $e (v, V (G)
  \backslash V (H_{\frac{n}{4}})) =3$ holds for each $v \in V (H_{\frac{n}{4}})$.
  Recall that $k\leqslant \left\lceil \frac{n}{4} \right\rceil$ for this case and so $k<\frac{n}{4}+1$. Thus, (b) cannot occur.

  Suppose that equalities hold, and hence (a) holds.
  By
  \eqref{NeighLess1}, for each $v \in V (G \backslash H_4)$, there exists a
  unique vertex $v' \in H_4$ such that $v$ and $v'$ are adjacent. Thus, we can
  divide the vertex set $V (G) \backslash V (H_4)$ into four subsets $V_1,
  V_2, V_3, V_4$ such that every subset $V_i$ is the set of neighbors of a
  vertex $h_i$ in $H_4$ ($i = 1, 2, 3, 4$) (see Fig. \ref{FourNeighbor}).
  Then the order of every subset is $\frac{n}{4} - 1$. Let $w_1 \in V_1$. Then we
  can compute $e (w_1, V (G) \backslash ( V(V_1) \cup V (H_4))
  \geqslant \frac{n}{4} + 2 - \left( \frac{n}{4} - 2 \right) - 1 = 3$; thus, there exists at least a neighbor $w_2$ of $w_1$ in $G-(V(H_4)\cup V_1)$. Without
  loss of generality, we can assume that $w_2 \in V_2$. Note that $d_{H_{_4}}
  (h_i) = 3$ for each $h \in \{ 1, 2, 3, 4 \}$, so $H_4 \cong K_4$. Therefore,
  $G \left[ V \left( {H_4}  \right) \cup \{ w_1, w_2 \} \backslash \{h_3\}
  \right]$ is a $2$-connected subgraph of order $5$.
  
  \begin{figure}[ht]
    \centering
  \raisebox{-0.4747405480575104\height}{\includegraphics[scale=0.7]{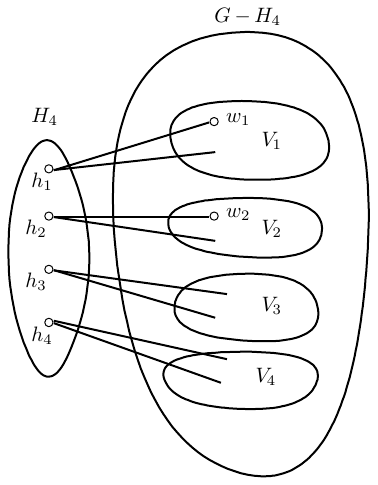}}
    \caption{Illustration for Case \ref{Case:1}}
    \label{FourNeighbor}
  \end{figure}
  
  For $6 \leqslant k \leqslant \left\lceil \frac{n}{4} \right\rceil$, $(k - 5) \left( k
  - \frac{n}{4} - 1 \right) < 0$, and hence $e (V (H_{k - 1}), V (G) \backslash V
  (H_{k - 1})) > n - k + 1$, a contradiction to (\ref{EdgeUpper}). It follows that there exists a
  vertex $x \in V (G) \backslash V (H_{k - 1})$ such that $| N_G (x) \cap V
  (H_{k - 1}) | > 2$. This enables us to find a $2$-connected subgraph of
  order $k$ in $G$.
  
  \begin{case}
    $\left\lceil \frac{n}{4} \right\rceil + 1 \leqslant k \leqslant \left\lceil \frac{n}{2}
    \right\rceil$.
  \end{case}
  Let $G'=G-V(H_{k-1})$.  The assumption $k \geqslant \left\lceil \frac{n}{4} \right\rceil + 1$ gives $| V (G')|\leqslant \frac{3 n}{4}$. By \eqref{MinDegHk-1}, we have $\delta (G') \geqslant \left\lceil \frac{n}{4} \right\rceil + 1.$ It follows that $\delta(G')\geqslant \frac{|V(G')|}{3}+1$. Recall $n\geqslant 16$, and this implies $|V(G')|\geqslant \left\lceil \frac{n}{4} \right\rceil +2>4$. By
  Theorem \ref{MinDegn3+1}, there is a $2$-connected subgraph of
  order $s$ in $G'$
  for each integer $s\in [4,|V(G')|]$.
Observe that $n-k+1\geqslant k$ as $k\leqslant \left\lceil \frac{n}{2}\right\rceil$. We know $G'$ contains a 2-connected subgraph of order $k$, which is also in $G$.

\begin{case}\label{Case:3}
$k = \left\lceil \frac{n}{2} \right\rceil + 1$.
\end{case}
  
The assumption $k = \left\lceil \frac{n}{2} \right\rceil + 1$ gives $\left| V (G)
  \backslash V (H_{k - 1}) = \left\lfloor \frac{n}{2} \right\rfloor \right|$. A simple
  computation shows that $N_{G - H_{k - 1}} (v) \cap N_{G - H_{k - 1}} (v')
  \neq \varnothing$ for any two different vertices $v, v' \in V (G) \backslash
  V (H_{k - 1})$. Let $V' = \{ v \in V (H_{k - 1}) : d_{H_{k - 1}} (v) < d_G
  (v) \}$. Let $a$ and $b$ be two vertices in $V'$. Select two different
  vertices $a' \in N_G (a) \backslash V (H_{k - 1})$ and $b' \in N_G (b)
  \backslash V (H_{k - 1})$ and let $c'$ be a vertex in $N_{G - V (H_{k - 1})}
  (a) \cap N_{G - V (H_{k - 2})} (b)$.
  
Let us consider two cases in terms of the order of $| V' |$.
  
First, suppose that $2 \leqslant | V' | \leqslant 3$.
Since $k = \left\lceil \frac{n}{2} \right\rceil + 1$, $| V (H_{k - 1}) | = \left\lceil
  \frac{n}{2} \right\rceil$. Thus, for any vertex $v \in V (H_{k - 1}) \backslash
  V'$, we have
  \[ d_{H_{k - 1}} (v) = d_G (v) \geqslant \left\lceil \frac{n}{4} \right\rceil + 2 =
     \frac{2 \left\lceil \frac{n}{4} \right\rceil + 4}{2} \geqslant \frac{\left\lceil
     \frac{n}{2} \right\rceil + 4}{2} > \frac{| V (H_{k - 1}) | + 3}{2} . \]
  By Lemma \ref{Removable}, there exist two vertices $v_1, v_2 \in V (H_{k -
  1})$ such that $| V' \backslash \{ v_1, v_2 \} | \geqslant 2$ and $H' = H_{k
  - 1}-\{v_1,v_2\}$ satisfy one of the conclusion
 (i)(ii) and (iii) (see Lemma \ref{Removable}). Then, $H_k = G [V (H') \cup
  \{ a', b', c' \}]$ is a $2$-connected subgraph of $G$ with order $k$.
  
Now suppose that
$| V' | \geqslant 4$.
If there exist two vertices $v_1, v_2$ such that $H' = H_{k - 1}-\{v_1,v_2\}$
  is still $2$-connected. Then $H_k = G [V (H_{k - 1}) \cup \{ a', b', c' \}
  \backslash \{ v_1, v_2 \}]$ is a $2$-connected subgraph of $G$ with order $k$.
  Otherwise, either $H_{k - 1}$ is critically $2$-connected, or there exists a
  vertex $x \in V (H_{k - 1})$ such that $H_{k - 1} - x$ is critically
  $2$-connected. By Corollary \ref{CritMin}, there exist two vertices $x, y$
  in $H_{k - 2}$ with a degree of at most $3$. Since $n \geqslant 16$, $\delta (G)
  \geqslant \left\lceil \frac{n}{4} \right\rceil + 2 \geqslant 6$, implying that $\{ x, y
  \} \subseteq V'$. Therefore, for each $v \in \{ x, y \}$, $| N_G (v)
  \backslash V (H_{k - 1}) | \geqslant \left\lceil \frac{n}{4} \right\rceil - 1$. By
  \eqref{EdgeUpper} and $k = \left\lceil \frac{n}{2} \right\rceil + 1$, we have
  \[ \left\lfloor \frac{n}{2} \right\rfloor \geqslant e (V', V (G) \backslash V (H_{k -
     1})) \geqslant 2 \left( \left\lceil \frac{n}{4} \right\rceil - 1 \right) + 2 = 2
     \left\lceil \frac{n}{4} \right\rceil . \]
  If $n \bmod 4 \neq 0$, then $2 \left\lceil \frac{n}{4} \right\rceil > \left\lfloor
  \frac{n}{2} \right\rfloor$, a contradiction. Now assume that $n \bmod 4 \equiv 0$,
  then equality holds if and only if $| V' | = 4$, for each $v \in V' \backslash
  \{ x, y \}$, $d_{G - H_{k - 1}} (v) = 1$, and for each $v \in \{ x, y \}$,
  $d_{G - H_{k - 1}} = \frac{n}{4} - 1$ (see Figure \ref{n/2+1}).
  
  \begin{figure}[ht]
    \centering
  \raisebox{-0.4747405480575104\height}{\includegraphics[scale=0.7]{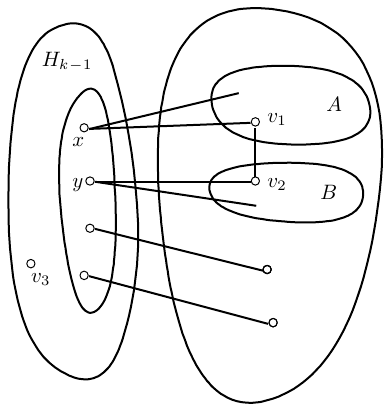}}
    \caption{Illustration for Case \ref{Case:3}}
    \label{n/2+1}
  \end{figure}
  
Let $A = N_{G - H_{k - 1}} (x)$ and $B = N_{G - H_{k - 1}} (y)$. For $v_1 \in A$, $| N_B (v_1) | \geqslant \frac{n}{4} + 2 - 1 - 2-\left( \frac{n}{4} -2\right)=1$. Then we can choose two vertices $v_2 \in N_B (v_1)$ and $v_3\in H_{k - 1} \backslash \{x,y\}$. Since $H_{k-1}$ is $2$-connected,
$H_{k-1}-v_3$ is connected. If $H_{k -1}-v_3$ is $2$-connected, then
$G [V (H_{k - 1}) \cup \{ v_1, v_2 \} \backslash v_3]$ is a $2$-connected
subgraph of $G$ with order $k$. So we suppose that $H_{k - 1}-v_3$ is not
$2$-connected.
  
\begin{claim}\label{V3OneBlock}
All vertices of $H_{k-1}-\{x,y\}$ are contained in the same block of $H_{k-1}-v_3$.\end{claim}
  
If not, assume that $x_1,y_1 \in V (H_{k - 1}) \backslash \{ x,y\}$
belong to different blocks of $H_{k-1} - v_3$. Hence, $| N_{H_{k-1}-v_3}(x_1) \cap N_{H_{k-1}-v_3} (y_1) | \leqslant 1$. Combining this
  with the fact that $$| N_{H_{k - 1} - v_3} (x_1) \cup N_{H_{k - 1} - v_3}
  (y_1) | + | N_{H_{k - 1} - v_3} (x_1) \cap N_{H_{k - 1} - v_3} (y_1) | =
  d_{H_{k - 1} - v_3} (x_1) + d_{H_{k - 1} - v_3} (y_1) \geqslant
  \frac{n}{2},$$ we have $| N_{H_{k - 1} - v_3} (x_1) \cup N_{H_{k - 1} - v_3}
  (y_1) | \geqslant \frac{n}{2} - 1$. However, $N_{H_{k - 1} - v_3} (x_1) \cup
  N_{H_{k - 1} - v_3} (y_1) \subseteq V (H_{k - 1}) \backslash \{ v_3, x_1,
  y_1 \}$, and thus $| N_{H_{k - 1} - v_3} (x_1) \cup N_{H_{k - 1} - v_3} (y_1)
  | \leqslant \frac{n}{2} - 3$, a contradiction.
  
  Therefore, by Claim \ref{V3OneBlock}, $H_{k - 1} - v_3$ has at most three
  blocks, and every cut vertex does not belong to $\{ x, y \}$. Thus, $G [V (H_{k
  - 1}) \cup \{ v_1, v_2 \} \backslash v_3]$ is a $2$-connected subgraph of
  $G$ with order $k$.
  
  \begin{case}
    $\left\lceil \frac{n}{2} \right\rceil + 2 \leqslant k \leqslant \left\lfloor \frac{3 n}{4}
    \right\rfloor-2$.
  \end{case}
  
  Since $k \geqslant \left\lceil \frac{n}{2} \right\rceil + 2$, one has $| V (G)
  \backslash V (H_{k - 3}) | \leqslant \left\lfloor \frac{n}{2} \right\rfloor + 1$. For
  any two different vertices $v, v' \in V (G) \backslash V (H_{k - 3})$ and $|
  N_{G - H_{k - 3}} (v) \cap N_{G - H_{k - 3}} (v') | = d_{G - H_{k - 3}} (v)
  + d_{G - H_{k - 3}} (v') - | N_{G - H_{k - 3}} (v) \cup N_{G - H_{k - 3}}
  (v') | \geqslant 2 \left( \left\lceil \frac{n}{4} \right\rceil + 1 \right) - \left\lfloor
  \frac{n}{2} \right\rfloor + 1 = 2 \left\lceil \frac{n}{4} \right\rceil - \left\lfloor \frac{n}{2}
  \right\rfloor + 1 \geqslant 1$. Let $V' = \{ v \in V (H_{k - 3}) : d_{H_{k - 3}}
  (v) < d_G (v) \}$. Let $a$ and $b$ be two vertices in $V'$. Select two
  different vertices $a' \in N_G (a) \backslash V (H_{k - 3})$ and $b' \in N_G
  (b) \backslash V (H_{k - 3})$, and let $c'$ be a vertex in $N_{G - V (H_{k -
  3})} (a) \cap N_{G - V (H_{k - 3})} (b)$. Then, $H_k = G [V (H_{k - 3}) \cup
  \{ a', b', c' \}]$ is a $2$-connected subgraph of $G$ with order $k$.

  \begin{case}
    $k \geqslant \left\lfloor \frac{3 n}{4} \right\rfloor - 1$.
  \end{case}
  
  Since $\delta (G) \geqslant \left\lceil \frac{n}{4} \right\rceil + 2$ and $| V (G)
  \backslash V (H_{k - 1}) | \leqslant \left\lceil \frac{n}{4} \right\rceil + 1$, each
  vertex $v \in V (G) \backslash V (H_{k - 1})$ must have at least two
neighbors that belong to $V (H_{k - 1})$ in $G$, contradicting our initial assumption.

The proof is complete.
\end{proof}

Note that $\delta (Q_3) = \frac{| V (Q_3) |}{4} + 1$. Thus, the condition of
Theorem \ref{MyMinDeg} is sharp when $n$ is small.

\section{Concluding remarks} \label{ExampleForTheMin}

In this section, we construct some examples to show that the condition of
Theorem \ref{MinDegn3+1} is sharp when $n$ is small. The first example is
$C_5$. we know that $\delta (C_5) = \left\lceil \frac{5}{3} \right\rceil = 2$, and $C_5$
does not contain any $2$-connected subgraph of order $4$. The second example is Hanoi graph $H_3^2$, shown in Fig.
\ref{GraphNineVertices}. The order of $H_3^2$ is $9$, and $\delta (H_3^2) =
\left\lceil \frac{9}{3} \right\rceil - 1 = 2$. Since $H^2_3$ does not contain $C_4$,
$H^2_3$ does not contain any $2$-connected subgraph of order $4$. It is
easy to check that $H^2_3$ contains no any $2$-connected subgraph of
order $5$.

\begin{figure}[ht]
  \centering
  \raisebox{-0.4747405480575104\height}{\includegraphics[scale=0.9]{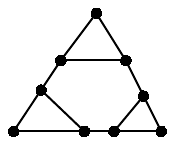}}
  \caption{The Hanoi graph $H_3^2$}
  \label{GraphNineVertices}
\end{figure}

Consider the following graph $H$ introduced by
Gould et al. \cite{GouldHaxell-701} in 2002 when they studied the problem of cylce spectra of a graph with large minimum degree. Let $m \geqslant 3$ be an
odd integer and suppose $n = 2ms$ where $s \geqslant 2$ is an integer.
We form a graph $H$ from a disjoint union of $m$ copies $K_1 [X_1, Y_1],
\ldots, K_m [X_m, Y_m]$ of the complete bipartite graph $K_{s, s}$ by adding
$m$ vertex-disjoint edges $e_1, \ldots, e_m$ as follows. For $1 \leqslant i
\leqslant m - 1$, we let $e_i$ join a vertex $y_i$ of $Y_i$ to a vertex $x_{i +
1}$ of $X_{i + 1}$, and we let $e_m$ join a vertex $y_m$ of $Y_m$ to a vertex
of $Y_m$ and to a vertex $y_0$ of $Y_1$ different from $y_1$ (see Fig.
\ref{GraphByKm}).

\begin{figure}[ht]
  \centering
  \raisebox{-0.4747405480575104\height}{\includegraphics[scale=0.7]{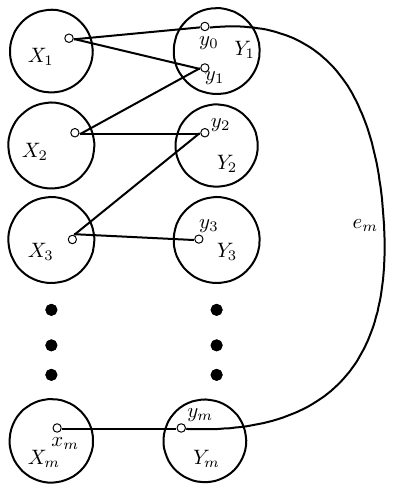}}
  \caption{A figure for Theorem \ref{Thm:4.1}}
  \label{GraphByKm}
\end{figure}

It is easily seen that $\delta (H) = s = \frac{n}{2 m}$. Now we consider all
the orders of $2$-connected subgraphs of $H$.

\begin{thm}\label{Thm:4.1}
  $H$ contains a $2$-connected subgraph of order $k$ only for each $k \in \{ 4, 5,
  \ldots, 2 s \} \cup \{ 2 m + 1, 2 m + 2, \ldots, n \}$.\label{SubgraphOfH}
\end{thm}

\begin{proof}
Let $H'$ be a $2$-connected subgraph of $H$.
  
  \begin{claim}
    If $H'$ has at least one edge of $\{ e_1, \ldots, e_m \}$, then $H'$
    contains all edges of $\{ e_1, \ldots, e_m \}$.\label{ContainEdge}
  \end{claim}
  
  Suppose, by way of contradiction, that $H'$ contains $e_i$ but no
  $e_{i + 1}$ for $i \in \{ 1, \ldots, m \}$(if $i = m$, let $e_{m + 1} =
  e_1$). Then $y_i$ is a cut vertex of $H'$, a contradiction.
  
  By Claim \ref{ContainEdge},~we distinguish two cases.
  \setcounter{case}{0}
  \begin{case}
    $H'$ does not have any edge in $\{ e_1, \ldots, e_m \}$.
  \end{case}
  
  In this case, $H'$ is a subgraph of $K_{s, s}$. Thus the order of $H'$ can
  be taken in $\{ 4, \ldots, 2 s \}$.
  
  \begin{case}
    $H'$ contains all edges of $\{ e_1, \ldots, e_m \}$.
  \end{case}
  
  In this case, $H'$ contains all vertices incident with edges in $\{ e_1,
  \ldots, e_m \}$ and a vertex of $X_1$. Thus the order of $H'$ is at least $2
  m + 1$. If the order of $H'$ is $2 m + 1$, then $$H' \cong H [\{ y_0, y_{1,}
  \ldots, y_m, x_1, \ldots, x_m \}]$$ where $x_1 \in X_1$. Let $\ell \in [2 m +
  2, n]$ be a integer. Since $K_1 [X_1, Y_1], \ldots, K_m [X_m, Y_m]$ are the
  complete bipartite subgraphs of $H$, we can form a $2$-connected subgraph
  $H_{\ell} \cong H [V (H') \cup W]$ of order $\ell$ in $H$ from $H'$ and a
  vertex subset $W \subseteq V (H) \backslash V (H')$ of order $\ell - 2 m -
  1$. Thus the order of $G'$ can be taken in $\{ 2 m + 1, 2 m + 2, \ldots, n
  \}$
  
  Therefore, $H$ contains a $2$-connected subgraph of order $k$ for each $k \in
  \{ 4, \ldots, 2 s \} \cup \{ 2 m + 1, \ldots, n \}$.
\end{proof}

Using Theorem \ref{SubgraphOfH}, it is easy to check that $H$ does not contain
$2$-connected subgraph of order $k$ for each $k \in \{ 2s+1, 2s+2 \}$ if $\delta (H) =\frac{n}{2(s+1)}$ when $s\geqslant 2$ and $n=2s(s+1)$.  
 Thus, if $\delta
  = \frac{n}{k}$ where $k$ is even and $k\geqslant 6$, then there exists a graph with minimum degree $\delta$ such that
  $G$ does not contain $2$-connected subgraphs of order $\ell$ for some $\ell \in \{ 4, \ldots, n
  \}$ if $n=2s(s+1)$.

It is conjectured that there is only finitely many $(v,k,2)-$ SBISDs (see 6.1.3 in \cite[pp.~111]{CD06}).  If the conjecture were to be true, then we cannot construct infinite counterexamples to Conjecture \ref{Conj:2}  by SBISDs. On the other hand, we still strongly suspect that Conjecture \ref{Conj:2} is false for infinite values of $n$. Therefore, we pose the
following conjecture, which is weaker.

\begin{conj}
  Let $G$ be a $2$-connected graph of order $n$. If $\delta (G)
  \geqslant \frac{n}{k}$ where $k\geqslant 3$, then there exists a integer $n_0=f(k)$ such that
  $G$ has $2$-connected subgraphs of order $\ell$ for each $\ell \in \{ 4, \ldots, n
  \}$ if $n \geqslant n_0$.
\end{conj}

\section*{Acknowledgment}
The second author thanks Yue Zhou for introducing the reference book \cite{{CD06}} to him.

\bibliographystyle{siam}
\bibliography{reference.bib}




\end{document}